\documentclass{amsart}
\usepackage{amsmath}%
\usepackage{amsfonts}%
\usepackage{amssymb}%
\usepackage{amsthm}
\usepackage{units}
\usepackage{enumerate} 
\usepackage{upgreek}
\usepackage{graphicx}
\usepackage{multicol}
\usepackage[usenames,dvipsnames]{xcolor}

\usepackage{array} 
\usepackage{tabulary} 
\usepackage{multirow} 

\numberwithin{equation}{section}

\usepackage{tikz}
\usepackage{tikz-cd}
\usetikzlibrary{matrix,arrows,decorations.pathmorphing}
\usepackage{hyperref} 
\usepackage{geometry}
\colorlet{ColRecap}{Black} 
\colorlet{ColDefB}{Black} 
\colorlet{ColRmkB}{Black} 
\colorlet{ColExB}{Black} 
\colorlet{ColThmB}{Black} 
\colorlet{ColApp}{Black} 

\colorlet{ColThm}{Black}
\colorlet{ColDef}{Black}
\colorlet{ColRmk}{Black}
\colorlet{ColEx}{Black}
\colorlet{ColRcp}{Black}
\colorlet{ColPrel}{Black}
\colorlet{ColSA}{Black}
\colorlet{ColAss}{Black}

\DeclareMathAlphabet{\mathpzc}{OT1}{pzc}{m}{it}

	\newcommand{\Omhol}{\Upomega} 

\DeclareMathOperator{\Id}{Id}

\DeclareMathOperator{\rk}{rk}
\DeclareMathOperator{\pr}{pr}

\DeclareMathOperator{\Div}{Div}

\DeclareMathOperator{\M}{M}

\DeclareMathOperator{\Hom}{Hom}  
\DeclareMathOperator{\End}{End}
\DeclareMathOperator{\Iso}{Iso}
\DeclareMathOperator{\Aut}{Aut}

\DeclareMathOperator{\Ob}{Ob}
\DeclareMathOperator{\Mor}{Mor}

\DeclareMathOperator{\MOR}{MOR}  

\DeclareMathOperator{\op}{op}
\DeclareMathOperator{\Rep}{\mathbf{Rep}} 
\DeclareMathOperator{\Coh}{\mathbf{Coh}} 

\DeclareMathOperator{\Ext}{Ext}

\DeclareMathOperator{\Bl}{Bl} 

\DeclareMathOperator{\Pic}{Pic}

\DeclareMathOperator{\num}{num}
\DeclareMathOperator{\cc}{c} 

\DeclareMathOperator{\GL}{GL}

\theoremstyle{plain}
\newtheorem{theorem}{Theorem}[section]
\newtheorem{lemma}[theorem]{Lemma}
\newtheorem{corollary}[theorem]{Corollary}
\newtheorem{proposition}[theorem]{Proposition}

\theoremstyle{definition}
\newtheorem{definition}[theorem]{Definition}
\newtheorem{remark}[theorem]{Remark}
\newtheorem{remarks}[theorem]{Remarks}
\newtheorem{example}[theorem]{Example}

\newcommand{\rroangle}{\rotatebox[origin=c]{-90}{$\angle$}}
\newcommand{\roangle}{\rotatebox[origin=c]{-48}{$\angle$}}

\begin{document}

%
%

\title{Linear data for framed sheaves via exceptional sequences}
\author{Andrea Maiorana}

\begin{abstract}
Let $X$ be a smooth projective surface with a full strong exceptional sequence $\mathfrak{E}$. Under certain conditions, we describe the moduli spaces of framed sheaves on a line in $X$ via linear data, i.e.\ by realizing them as principal bundles over a stack of representations of the bound quiver associated to $\mathfrak{E}$.
\end{abstract}

\maketitle

%
%
\section{Introduction}
\subsection{Background}\label{Background}
In the whole paper, $X$ will denote a smooth projective irreducible complex surface and $C_0\subset X$ a smooth curve.

We are interested in studying moduli spaces of \emph{framed sheaves} on $(X,C_0)$, that is couples $(\mathcal{E},\phi)$ consisting of a torsion-free sheaf $\mathcal{E}$ on $X$ and a \emph{framing}, an isomorphism $\phi:\mathcal{E}{\restriction_{C_0}}\to\mathcal{V}$ to a fixed vector bundle on $C_0$ which in this paper will always be the trivial bundle. We denote by $\M^\textup{fr}_{X,C_0}(v)$ the moduli space of framed sheaves whose numerical invariants $v\in K_{\num}(X)$ (i.e.\ rank and Chern classes) are fixed, which was shown in \cite{HuyLeh95Framed,BruMar11Mod} to exist as a quasi-projective scheme under mild conditions on $C_0$. These moduli spaces often provide desingularizations of spaces of ideal instantons, which makes them interesting for applications in topological field theories. Moreover, they reduce as special cases to Hilbert schemes of points on the open variety $X\setminus C_0$.

In this paper we are concerned with the approach via \emph{linear data}, or \emph{ADHM data}, to the construction of $\M^\textup{fr}_{X,C_0}(v)$, which essentially consists, for suitable choices of $X$ and $C_0$, in representing it as a space of matrices (in fact, of representations of a quiver) up to the action of a linear algebraic group. This method, originally developed for $X=\mathbb{P}^2$ in \cite{Don84Ins,Nakaj99Lect}, has been later carried out in \cite{King89Inst,Henni14Monads,BaBrRa15Monad} to describe moduli of framed sheaves on blow-ups of $\mathbb{P}^2$ and Hirzebruch surfaces $\Sigma_e$, $e>0$. In each case the strategy was the following:
	\begin{enumerate}
	\item First, one shows that each torsion-free sheaf $\mathcal{E}$ trivial on $C_0$ satisfies some cohomological vanishings, which imply that a certain generalized Beilinson spectral sequence degenerates, giving a quasi-isomorphism between $\mathcal{E}$ and a complex $M^{-1}\to M^0\to M^1$ of vector bundles (called a \emph{monad}), where each $M^i$ only depends on the numerical invariants of $\mathcal{E}$.
	\item By using standard techniques on monads (see e.g.\ \cite[\S2.3-2.4]{OSS80Ve}) isomorphism classes of such sheaves are identified with isomorphism classes of their monads, and hence with the points of a certain subvariety $V\subset\Hom(M^{-1},M^0)\times\Hom(M^0,M^1)$ up to the action of $G:=\prod_i\Aut(M^i)$; typically these Hom-spaces are simple enough that $V$ can be naturally seen as a variety of linear maps between vector spaces. After including the framings in the data, $\M^\textup{fr}_{X,C_0}(v)$ is realized as a free quotient $\tilde{V}/G$.
	\end{enumerate}

\subsection{Outline and main results of the paper}
The goal of this paper is to reconsider this strategy from a deeper viewpoint that allows to construct the linear data description systematically, instead of focusing on a case-by-case analysis. This is suggested by the observation that the linear data that one gets can sometimes be seen as representations of the bound quiver $(Q,J)$ determined by a strong exceptional sequence $\mathfrak{E}$ in $D^b(X)$, and the existence of the generalized Beilinson spectral sequence mentioned before is a consequence of a derived equivalence $D^b(X)\simeq D^b(Q;J)$ induced by $\mathfrak{E}$. This approach was also used in \cite{Maior17Mod} to describe moduli spaces of Gieseker-semistable sheaves.

After reviewing the basics of moduli spaces of framed sheaves in \S\ref{Mod fr shvs}, we proceed as follows: first, in \S\ref{Framable shvs} we introduce an auxiliary algebraic stack $\mathfrak{M}_{X,C_0}^\textup{t}(v)$ which parameterizes torsion-free sheaves that are trivial on $C_0$ (which now on is assumed to be isomorphic to $\mathbb{P}^1$) and we show that the morphism $\M^\textup{fr}_{X,C_0}(v)\to\mathfrak{M}_{X,C_0}^\textup{t}(v)$ forgetting the framing is a principal $\GL(\rk v,\mathbb{C})$-bundle over it (principal bundles over stacks are discussed in Appendix \ref{App: prin bun on stack}).

At this point we focus on describing $\mathfrak{M}_{X,C_0}^\textup{t}(v)$ using the above-mentioned equivalence $\Psi:D^b(X)\simeq D^b(Q;J)$ (this machinery is briefly reviewed in \S\ref{Exc seq quiv mod}): in particular, we can use the stack $\mathfrak{M}_{Q,J}(d^v)$ of representations of $(Q,J)$ of fixed dimension vector $d^v$ (corresponding to the class $v$ via $\Psi$). Recall that this stack is a quotient $[X_{d^v,J}/G_{d^v}]$ of an affine scheme $X_{d^v,J}$ parameterizing $d^v$-dimensional representations of $(Q,J)$ by the natural action of the symmetry group $G_{d^v}$, which is a product of general linear groups.

Our goal of ``describing $\M_{X,C_0}^\textup{fr}(v)$ via linear data'' will now consist in embedding $\mathfrak{M}_{X,C_0}^\textup{t}(v)$ as an open substack $[U_v/G_{d^v}]\subset\mathfrak{M}_{Q,J}(d^v)$. This implies that $\M^\textup{fr}_{X,C_0}(v)$ inherits its local geometric properties from those of the better understood $\mathfrak{M}_{Q,J}(d^v)$. To show that this embedding exists (which, in the framework of the previous subsection, corresponds to showing the degeneration of the generalized Beilinson spectral sequence) we need some general cohomological vanishings for torsion-free sheaves trivial on the line $C_0$, which are proven in Prop.\ \ref{Prop: van fr sh mov lin}. Putting all together, we can summarize the main results as follows:

\begin{theorem}\label{Thm: main thm}
Suppose that $D^b(X)$ has a full strong exceptional sequence
\begin{equation*}
\mathfrak{E}=(\mathcal{O}_X(D_n),...,\mathcal{O}_X(D_1),\mathcal{O}_X(D_0))\,,
\end{equation*}
consisting of line bundles. Let $C_0\subset X$ be big, nef and isomorphic to $\mathbb{P}^1$, 
and assume that there exists a line bundle $\mathcal{L}\in\Pic X$ such that $\cc_1(\mathcal{L})\cdot C_0=1$. Suppose further that for any $i=0,...,n$ we have
\begin{equation*}
0<D_i\cdot C_0<-K_X\cdot C_0\,,
\end{equation*}
where $K_X$ is the canonical divisor on $X$. Then:
	\begin{enumerate}
	\item $\mathfrak{M}^\textup{t}_{X,C_0}(v)$ is isomorphic to an open substack $[U_v/G_{d^v}]\subset[X_{d^v,J}/G_{d^v}]=\mathfrak{M}_{Q,J}(d^v)$;
	\item $\M^\textup{fr}_{X,C_0}(v)$ is smooth and it is a free quotient $P_v/G_{d^v}$, where $P_v$ is a principal $\GL(\rk v,\mathbb{C})$-bundle over $U_v$ carrying a lift of the action $G_{d^v}\curvearrowright X_{d^v,J}$.
	\end{enumerate}
\end{theorem}

Many rational surfaces possess strong exceptional sequences of line bundles. Using this fact, in \S\ref{Examples} we will apply the Theorem to concrete cases, obtaining examples of constructions of linear data for of $\M^\textup{fr}_{X,C_0}(v)$. At this point, this is reduced to checking that $X$ and $C_0$ satisfy the hypotheses of the Theorem, and to computing $(Q,J)$ and $d^v$.

\subsection{Notation and conventions on sheaves and quiver representations}\label{Notation}
\subsubsection{Schemes and stacks}
All schemes and algebraic stacks are understood to be locally of finite type over $\mathbb{C}$. A ``point'' $s\in S$ in a scheme $S$ is always a closed point.

\subsubsection{Coherent sheaves}
$\Coh_{\mathcal{O}_T}$ denotes the abelian category of coherent sheaves on a scheme $T$ and $D^b(T)$ is its bounded derived category.\\
$X$ will be a smooth projective irreducible complex surface; its tangent and cotangent bundles are denoted $\tau_X$ and $\Omhol_X$, and $K_X$ is the canonical divisor. $K_{\num}(X)$ is the numerical Grothendieck group of $\Coh_{\mathcal{O}_X}$, and given a class $v\in K_{\num}(X)$, we denote by $\mathfrak{M}_X(v)$ the algebraic $\mathbb{C}$-stack of coherent sheaves on $X$: its objects are flat families $\mathcal{F}\in\Coh_{\mathcal{O}_{X\times S}}$ over a scheme $S$ and for any point $s\in S$ we write $\mathcal{F}_s:=\iota_s^*\mathcal{F}$, where $\iota_s:X\to X\times S$ is the embedding $x\mapsto(x,s)$.

\subsubsection{Quiver representations}
$Q$ will always denote an oriented quiver, with vertices labelled by $I=\{0,...,n\}$, $\mathbb{C}Q$ its complex path algebra, $J\subset\mathbb{C}Q$ an ideal of relations (i.e.\ one generated by paths of lenght $\geq 2$). $\Rep^\textup{fd}_\mathbb{C}(Q;J)$ is the abelian category of finite-dimensional complex representations of $Q$, bound by the relations $J$, and $D^b(Q;J)$ is its bounded derived category. Given a dimension vector $d\in\mathbb{N}^I$, the affine space $R_d:=\oplus_{i\to j}\Hom_\mathbb{C}(\mathbb{C}^{d_i},\mathbb{C}^{d_j})$ parameterizes $d$-dimensional representations of $Q$, and the relations $J$ define a closed subscheme $X_{d,J}\subset R_d$. The group $G_d:=\prod_{i\in I}\GL(\mathbb{C},d_i)$ acts on $R_d$ by simultaneous conjugation, and preserves $X_{d,J}$. The orbits of this action correspond to isomorphism classes of $d$-dimensional representations. 

Finally, if $S$ is a scheme, we denote by $\Coh_{\mathcal{O}_S}^{(Q,J)}$ the abelian category of coherent $\mathcal{O}_S$-modules with a left action of $\mathbb{C}Q/J$. Recall that a flat family over $S$ of representations of $(Q,J)$ is defined as a locally free $\mathcal{V}\in\Coh_{\mathcal{O}_S}^{(Q,J)}$. Flat families of representations of fixed dimension vector $d\in\mathbb{N}^I$ make an algebraic $\mathbb{C}$-stack $\mathfrak{M}_{Q,J}(v)$ isomorphic to the quotient $[X_{d,J}/G_d]$.

%
%
\section{Moduli of framed sheaves}
\subsection{Framed sheaves and their moduli}\label{Mod fr shvs}
Fix a curve $C_0\subset X$.

\begin{definition}
A \emph{framed sheaf} is a pair $(\mathcal{E},\phi)$ consisting of a coherent torsion-free sheaf $\mathcal{E}$ on $X$ and an isomorphism $\phi:\mathcal{E}{\restriction_{C_0}}\to\mathcal{O}_{C_0}^{\oplus\rk\mathcal{E}}$, called a \emph{framing}.
\end{definition}

More generally people consider sheaves with an isomorphism $\phi$ to a fixed vector bundle $\mathcal{V}$ on $C_0$, but in this paper we will only stick to the case in which $\mathcal{V}$ is trivial.

\begin{definition}
A \emph{morphism} between two framed sheaves $(\mathcal{E},\phi)$ and $(\mathcal{E}',\phi')$ is a morphism $\xi:\mathcal{E}\to\mathcal{E}'$ such that $\phi'\circ\xi{\restriction_{C_0}}=\phi$.\footnote{Notice that a slightly different definition is also common: some people only require $\xi$ to satisfy $\phi'\circ\xi{\restriction_{C_0}}=\lambda\phi$ for some $\lambda\in\mathbb{C}$. This allows homoteties to be automorphisms of a framed sheaves, while under some conditions our definition forces the automorphism group of $(\mathcal{E},\phi)$ o be trivial (see also Remark \ref{Rmk: conseq vanish Prop}). Hence with our choice the stack $\mathfrak{M}_{X,C_0}^\textup{fr}(v)$ (and not just its moduli functor) is representable.}
\end{definition}

\begin{remark}
Let $(\mathcal{E},\phi)$ be a framed sheaf. Since $\mathcal{E}$ is torsion-free, its singular locus (i.e.\ where its stalks are not free modules, or equivalently where its fibers have dimension equal to $\rk\mathcal{E}$) consists of finitely many points, disjoint from $C_0$. This implies in particular that:
	\begin{enumerate}
	\item $\mathcal{E}^\vee$ and $\mathcal{E}^{\vee\vee}$ are also trivial on $C_0$;
	\item $\cc_1(\mathcal{E})\cdot C_0=0$.
	\end{enumerate}
\end{remark}

\begin{definition}
Fix a class $v\in K_{\num}(X)$, and let $r:=\rk v$. A \emph{flat family} of framed sheaves of class $v$ over a scheme $S$ is a couple $(\mathcal{F},\phi)$, where $\mathcal{F}\in\Coh_{\mathcal{O}_{X\times S}}$ is $S$-flat, each $\mathcal{F}_s$ is a torsion-free sheaf of class $v$ and $\phi:\mathcal{F}{\restriction_{C_0\times S}}\to\mathcal{O}_{C_0\times S}^{\oplus r}$ is an isomorphism. An isomorphism between flat families $(\mathcal{F},\phi)$ and $(\mathcal{F}',\phi')$ is an isomorphism $\xi:\mathcal{F}\to\mathcal{F}'$ such that $\phi'\circ\xi=\xi$, and the pullback of $(\mathcal{F},\phi)$ along a morphism $f:S'\to S$ is obtained in the obvious way (and as usual we will write $f^*\mathcal{F}$ in place of $(\Id_X\times f)^*\mathcal{F}$). Such flat families form a stack, which we denote by $\mathfrak{M}_{X,C_0}^\textup{fr}(v)$.
\end{definition}

Using the more general theory of \emph{framed modules} developed in \cite{HuyLeh95Framed}, Bruzzo and Markushevich proved:

\begin{theorem}\cite[Cor.\ 3.3]{BruMar11Mod}\label{Thm: BM}
If $C_0$ is a smooth, irreducible, big and nef curve, then $\mathfrak{M}_{X,C_0}^\textup{fr}(v)$ is representable by a quasi-projective scheme $\M_{X,C_0}^\textup{fr}(v)$.
\end{theorem}

\subsection{Description of \texorpdfstring{$\mathfrak{M}_{X,C_0}^\textup{fr}(v)$}{Mfr(v)} via framable sheaves}\label{Framable shvs}
Now on we will focus on the case in which $C_0\simeq\mathbb{P}^1$. We fix $v\in K_{\num}(X)$ and denote $r:=\rk v$.

In this subsection, we will describe the stack $\mathfrak{M}_{X,C_0}^\textup{fr}(v)$ of framed sheaves via another stack, parameterizing sheaves which are trivial on $C_0$, but without the choice of a framing $\phi$:

\begin{definition}
The moduli stack $\mathfrak{M}_{X,C_0}^\textup{t}(v)$ is the substack of $\mathfrak{M}_X(v)$ whose objects over a scheme $S$ are $S$-flat sheaves $\mathcal{F}\in\Coh_{\mathcal{O}_{X\times S}}$ such that, for all $s\in S$, $\mathcal{F}_s$ is a torsion-free sheaf of class $v$, whose restriction $\mathcal{F}_s{\restriction_{C_0}}$ is trivial. We will call this the stack of \emph{framable sheaves}.
\end{definition}

\begin{remark}
Notice that if $\mathcal{F}\in\Coh_{\mathcal{O}_{X\times S}}$ is a flat family of framable sheaves, then the restriction $\mathcal{F}{\restriction_{C_0\times S}}$ is a locally free sheaf which is not necessarily trivial. Thus we cannot obtain a family of \emph{framed} sheaves simply adding to $\mathcal{F}$ the choice of a framing. However, we will see in a while that it is still possible to construct canonically a family of framed sheaves out of $\mathcal{F}$.
\end{remark}

\begin{lemma}\label{Lem: triv on P1 open cond}
$\mathfrak{M}_{X,C_0}^\textup{t}(v)\subset\mathfrak{M}_X(v)$ is an open substack.
\end{lemma}

In particular (see e.g.\ \cite[Prop.\ 4.5 and Thm.\ 4.6.2.1]{LauMor00Champs}) this tells us that $\mathfrak{M}_{X,C_0}^\textup{t}(v)$ is an algebraic $\mathbb{C}$-stack locally of finite type, because so is $\mathfrak{M}_X(v)$.

\begin{proof}
Being torsion-free is an open condition by \cite[Prop. 2.1]{Maruy76Openn}; being locally free on $C_0$ (which is equivalent to being locally free in some open neighbourhood of $C_0$) is also an open condition. So without loss of generality we can take an $S$-flat family $\mathcal{F}\in\Coh_{\mathcal{O}_{X\times S}}$ (of torsion-free sheaves) which is locally free on $C_0\times S$, i.e.\ which restricts to an $S$-flat family $\tilde{\mathcal{F}}$ of vector bundles on $\mathbb{P}^1$. Now we claim that the set
\begin{equation*}
U:=\{s\in S\ |\ \tilde{\mathcal{F}}_s\textup{ is trivial}\}
\end{equation*}
is open: indeed, if this set is nonempty, then all the vector bundles of the family have degree $\cc_1(v)\cdot C_0=0$, and for a degree zero vector bundle $\mathcal{E}$ on $\mathbb{P}^1$ being trivial is equivalent to satisfying $h^0(\mathbb{P}^1,\mathcal{E}(-1))=0$. So $U$ is open by the semicontinuity theorem.
\end{proof}

Now we explain the relation between $\mathfrak{M}_{X,C_0}^\textup{t}(v)$ and $\mathfrak{M}_{X,C_0}^\textup{fr}(v)$: we have a morphism
\begin{equation*}
\mathfrak{M}_{X,C_0}^\textup{fr}(v)\to\mathfrak{M}_{X,C_0}^\textup{t}(v)
\end{equation*}
forgetting the framing, whose fiber over a point $[\mathcal{E}]$ consists of all the possible framings $\phi\in\Hom(\mathcal{E}{\restriction_{C_0}},\mathcal{O}_X^{\oplus r})\simeq\GL(r,\mathbb{C})$. In the remainder of this subsection we will use the formalism discussed in the Appendix to argue that, as the intuition suggests, this morphism is a principal $\GL(r,\mathbb{C})$-bundle.

This fact is due to the following construction (which was essentially carried out in a special case in \cite[\S4.2]{BaBrRa15Monad}): given a scheme $S$ and family $\mathcal{F}\in\mathfrak{M}_{X,C_0}^\textup{t}(v)$ over $S$, we are going to build a principal $\GL(r,\mathbb{C})$-bundle $p:P_\mathcal{F}\to S$ and make the pullback $p^*\mathcal{F}$ naturally into a family of \emph{framed} sheaves. As already observed, the restriction $\mathcal{F}{\restriction_{C_0\times S}}$ is a (not necessarily trivial) vector bundle, trivial on each fiber of the projection $\pr_S$ onto $S$. Consider its pushforward
\begin{equation*}
\mathcal{V}_\mathcal{F}:=(\pr_S)_*\mathcal{F}{\restriction_{C_0\times S}}\,.
\end{equation*}

\begin{remark}\label{Rmk: cbc for VF}
For any $s\in S$, $h^0(C_0;\mathcal{F}_s{\restriction_{C_0}})=r$ and $h^i(C_0;\mathcal{F}_s{\restriction_{C_0}})=0$ for $i\neq 0$. Hence, by cohomology and base change we see that (having used, in particular, \cite[Cor.\ 7.9.9]{Groth63EGAIII2}):
	\begin{enumerate}
	\item $\mathcal{V}_\mathcal{F}$ is a vector bundle on $S$, whose fiber over a point $s\in S$ is the space $H^0(C_0;\mathcal{F}_s{\restriction_{C_0}})\simeq\mathbb{C}^r$;
	\item the formation of $\mathcal{V}_\mathcal{F}$ commutes with base change: given a morphism $f:S'\to S$, we have $\mathcal{V}_{f^*\mathcal{F}}\cong f^*\mathcal{V}_\mathcal{F}$.
	\end{enumerate}
\end{remark}

Now denote by
\begin{equation*}
p:P_\mathcal{F}\to S
\end{equation*}
the frame bundle of $\mathcal{V}_\mathcal{F}$: this is a principal $\GL(r,\mathbb{C})$-bundle whose fiber over $s\in S$ is the set $\Iso(\mathbb{C}^r,H^0(C_0;\mathcal{F}_s{\restriction_{C_0}}))\cong\Iso(\mathcal{O}_{C_0}^{\oplus r},\mathcal{F}_s{\restriction_{C_0}})$.

The reason we are interested in this principal bundle is that its sections describe framings of the family $\mathcal{F}$:

\begin{lemma}\label{Lem: triv bun of fr}
$P_\mathcal{F}$ (or $\mathcal{V}_\mathcal{F}$) is trivial if and only if $\mathcal{F}{\restriction_{C_0\times S}}$ is trivial. In this case, global sections of $p$ are identified with framings $\mathcal{F}{\restriction_{C_0\times S}}\to\mathcal{O}_{C_0\times S}^{\oplus r}$ via the canonical isomorphism $\Hom(\mathcal{O}_S^{\oplus r},\mathcal{V}_\mathcal{F})\cong\Hom(\mathcal{O}_{C_0\times S}^{\oplus r},\mathcal{F}{\restriction_{C_0\times S}})$.
\end{lemma}

\begin{proof}
Since $\pr_S$ is projective and has connected fibers, we have $(\pr_S)_*\mathcal{O}_{C_0\times S}^{\oplus r}\simeq\mathcal{O}_S^{\oplus r}$. Conversely, if $\mathcal{V}_\mathcal{F}$ is trivial, then we can take $r$ linearly independent sections (which are the same as a global section of $P_\mathcal{F}$): by triviality of $\mathcal{F}{\restriction_{C_0\times S}}$ on the fibers of $\pr_S$, these give $r$ pointwise independent sections of $\mathcal{F}{\restriction_{C_0\times S}}$, i.e.\ the inverse map of a framing $\mathcal{F}{\restriction_{C_0\times S}}\to\mathcal{O}_{C_0\times S}^{\oplus r}$.
\end{proof}

Now, for each family $\mathcal{F}\in\mathfrak{M}_{X,C_0}^\textup{t}(v)$ over a scheme $S$ we have constructed in a canonical way a principal $\GL(r,\mathbb{C})$-bundle $P_\mathcal{F}\to S$. Moreover, Remark \ref{Rmk: cbc for VF} says that this construction is compatible with pullbacks, so we have realized a functor
\begin{equation*}
\Phi:\mathfrak{M}_{X,C_0}^\textup{t}(v)\longrightarrow B\GL(r,\mathbb{C})\,,
\end{equation*}
i.e.\ a principal $\GL(r,\mathbb{C})$-bundle on $\mathfrak{M}_{X,C_0}^\textup{t}(v)$ (Def.\ \ref{Defn: prin G-bun stk}). Now, we can identify the total space $\mathfrak{P}$ of $\Phi$ (Def.\ \ref{Defn: tot sp pr bun stk}) with the stack $\mathfrak{M}_{X,C_0}^\textup{fr}(v)$: by definition, an object of $\mathfrak{P}$ over a scheme $S$ consists of a family $\mathcal{F}\in\mathfrak{M}_{X,C_0}^\textup{t}(v)$ over $S$ together with a section of the bundle $p:P_\mathcal{F}\to S$. By Lemma \ref{Lem: triv bun of fr}, such a section corresponds to a framing $\phi:\mathcal{F}\overset{\simeq}{\to}\mathcal{O}_{C_0\times S}^{\oplus r}$. Clearly this correspondence is functorial and compatible with pulling back families, and thus we have deduced:

\begin{proposition}\label{Prop: Mfr prin bun}
$\Phi$ is a principal $\GL(r,\mathbb{C})$-bundle, whose total space is identified with the forgetful morphism $\mathfrak{M}_{X,C_0}^\textup{fr}(v)\to\mathfrak{M}_{X,C_0}^\textup{t}(v)$. In particular, $\mathfrak{M}_{X,C_0}^\textup{fr}(v)$ is an algebraic $\mathbb{K}$-stack locally of finite type.
\end{proposition}

Concretely, this means that a family $\mathcal{F}\in\mathfrak{M}_{X,C_0}^\textup{t}(v)$ over a scheme $S$ induces a family of framed sheaves by pulling back along the forgetful morphism: if $\alpha_\mathcal{F}:S\to\mathfrak{M}_{X,C_0}^\textup{t}(v)$ is the morphism corresponding to $\mathcal{F}$, then we have a fiber product (see Lemma \ref{Lem: pullb fam prin bun stk})
\begin{center}\begin{tikzcd}
P_\mathcal{F}\arrow{r}\arrow{d}{p}\arrow[dr, phantom, "\ulcorner", very near start] &\mathfrak{M}_{X,C_0}^\textup{fr}(v)\arrow{d}\\
S\arrow{r}{\alpha_\mathcal{F}} &\mathfrak{M}_{X,C_0}^\textup{t}(v)
\end{tikzcd}\,.\end{center}
This gives in particular a morphism $P_\mathcal{F}\to\mathfrak{M}_{X,C_0}^\textup{fr}(v)$, i.e.\ a family of framed sheaves over $P_\mathcal{F}$. A simple inspection of the diagram above shows that this family is (up to isomorphism) the couple $(p^*\mathcal{F},\phi)$, where $\phi:p^*\mathcal{F}{\restriction_{C_0\times P_\mathcal{F}}}\to\mathcal{O}_{C_0\times P_\mathcal{F}}$ is the framing which corresponds via Lemma \ref{Lem: triv bun of fr} to the tautological section of $p^*P_\mathcal{F}\simeq P_{p^*\mathcal{F}}$.

%
%
\section{Linear data for framable sheaves}
In this section $C_0\subset X$ will denote a curve isomorphic to $\mathbb{P}^1$. The goal is to describe the stack $\mathfrak{M}_{X,C_0}^\textup{t}(v)$ via linear data, by embedding it as an open substack of the stack of representations of a quiver with relations.

\subsection{Exceptional sequences and quiver moduli}\label{Exc seq quiv mod}
Here we will assume that $D^b(X)$ admits a full strong exceptional sequence $\mathfrak{E}=(E_n,...,E_0)$, which (for simplicity) consists of vector bundles. Recall that by definition this means that the bundles $E_0,...,E_n$ generate the whole $D^b(X)$ and we have
\begin{equation*}
\Hom_{D^b(X)}(E_i,E_j[\ell])=0
\end{equation*}
for all $i<j$ and all $\ell\in\mathbb{Z}$, and also for all $i,j$ and $\ell\neq 0$. 

In this subsection we will gather a few known facts on how the sequence $\mathfrak{E}$ allows to identify certain ``flat families'' of objects in $D^b(X)$ with flat families of representations of a bound quiver, and then we will discuss how this is relevant to our goal of describing the stack $\mathfrak{M}_{X,C_0}^\textup{t}(v)$. Some references for the following material are \cite[\S3]{TodUeh10Tilt}, \cite[\S7.3]{BaCrZh17Nef}, \cite[\S2.2.2, \S4.2.4]{Maior18Mod}.

	\begin{enumerate}
	\item $\mathfrak{E}$ defines a so-called \emph{tilting bundle} $T:=\oplus_iE_i$ whose endomorphism algebra $\End(T)$ is basic, which means that we can write
	\begin{equation*}
	\End(T)^{\op}\simeq\End(T^\vee)\simeq\mathbb{C}Q/J\,,
	\end{equation*}
	where $Q$ is an ordered quiver with vertices labelled by $I:=\{0,...,n\}$ and $J\subset\mathbb{C}Q$ is an ideal of relations, such that the paths between the $i$th and $j$th vertices modulo the relations are indexed by a basis of $\Hom(E_j,E_i)$:
	\begin{equation*}
	\Hom(E_j,E_i)\simeq e_j\left(\nicefrac{\mathbb{K}Q}{J}\right)e_i\,.
	\end{equation*}	
	\item We have a triangulated equivalence 
	\begin{equation*}
	\Psi=R\Hom_{D^b(X)}(T,\cdot)[1]:D^b(X)\longrightarrow D^b(Q;J)
	\end{equation*}
	(the shift $[1]$ has been inserted for future convenience), sending an object $\mathcal{E}\in D^b(X)$ to a complex of representations given, at the $i$th vertex, by the graded vector space
	\begin{equation}\label{Eq: expl form der eq}
	R\Hom(E_i,\mathcal{E})[1]\,.
	\end{equation}
	\item The inverse image of $\Rep^\textup{fd}_\mathbb{C}(Q,J)$ under $\Psi$ is the heart of a bounded t-structure in $D^b(X)$, which we denote by
	\begin{equation*}
	\mathcal{K}\subset D^b(X)\,.
	\end{equation*}
	If we denote by $\mathfrak{E}^*=(E_0^*,...,E_n^*)$ the right dual sequence to $\mathfrak{E}$, i.e.\ the full exceptional sequence uniquely determined by the relations
		\begin{equation*}
		\Hom_{D^b(X)}(E_i,E^*_j[\ell])=\left\{\begin{matrix}
		\mathbb{C} &\textup{if}\ j=i\textup{ and }\ell=0\,, \\ 
		0 &\textup{otherwise} 
		\end{matrix}\right.\,,
		\end{equation*}
	then $\mathcal{K}$ coincides with the extension-closure of the objects $E^*_i[-1]$ for $i=0,...n$. Notice also that if $E_i^*[-k_i]$ is a sheaf for all $i$ and some $k_i\in\mathbb{Z}$, then for all $\mathcal{E}\in D^b(X)$ we have a \emph{generalized Beilinson spectral sequence}\footnote{For example, if $X=\mathbb{P}^2$ and $\mathfrak{E}=(\mathcal{O}(1),\uptau_{\mathbb{P}^2},\mathcal{O}(2))$, then $\mathfrak{E}^*=(\mathcal{O}_{\mathbb{P}^2}(-1)[2],\mathcal{O}_{\mathbb{P}^2}[1],\mathcal{O}_{\mathbb{P}^2}(1))$ and this reduces to the usual Beilinson spectral sequence
	\begin{equation*}
	E^{p,q}_1=\mathbb{H}^q(\mathcal{E}\otimes\Omhol^p(-p-1))\otimes\mathcal{O}(p+1)\Rightarrow H^{p+q}(\mathcal{E})\,.
	\end{equation*}
	Notice also that this does not require that $\mathfrak{E}$ be strong or consisting of vector bundles. 
	}
	\begin{equation}\label{Eq: gen Beil sp seq heart}
	E^{p,q}_1=\Hom_{D^b(X)}(E_{p+n},\mathcal{E}[p+q+k_p+n])\otimes E^*_p[-k_p]\Rightarrow H^{p+q}(\mathcal{E})
	\end{equation}
	mentioned in \S\ref{Background}.
	\item Given the heart $\mathcal{A}\subset D^b(X)$ of a bounded t-structure, we call \emph{flat family} of objects in $\mathcal{A}$ parameterized by a scheme $S$ an object $\mathcal{F}\in D^b(X\times S)$ such that for all $s\in S$ the derived pullback $\mathcal{F}_s:=L\iota_s^*\mathcal{F}\in D^-(X)$ along the embedding $\iota_s:x\mapsto(x,s)$ is an object of $\mathcal{A}$. According to this definition, a flat family of objects in the standard heart $\mathcal{C}=\Coh_{\mathcal{O}_X}$ is the same (i.e.\ it is quasi-isomorphic to) as a flat family of coherent sheaves in the usual sense.
	\item A similar argument can be done for the heart $\mathcal{K}$: for any scheme $S$ we have an equivalence
	\begin{equation}\label{Eq: der eq families}
	\Psi_S:D^b(X\times S)\longrightarrow D^b(\Coh_{\mathcal{O}_S}^{(Q,J)})
	\end{equation}
	(recall the notation introduced in \ref{Notation}) compatible with pullbacks, and reducing to the equivalence $\Psi$ when $S$ is a point; under $\Psi_S$, flat families of objects in $\mathcal{K}$ correspond to flat families of representations, i.e. locally free sheaves in $\Coh_{\mathcal{O}_S}^{(Q,J)}$. In this sense, given a class $v\in K_{\num}(X)\cong K_0(X)$ and denoted by $d^v\in\mathbb{N}^I$ the corresponding dimension vector under $\Psi$, that is
	\begin{equation}
	d^v_i=-\chi(E_i,v)\,,
	\end{equation}
	we can identify $\mathfrak{M}_{Q,J}(d^v)=[X_{d^v,J}/G_{d^v}]$ as the stack parameterizing the objects of class $v$ in $\mathcal{K}$.
	\end{enumerate}

Now we are interested in the objects lying in the intersection $\mathcal{C}\cap\mathcal{K}$: these are coherent sheaves on $X$ parameterized by points of the ``easier'' stack $\mathfrak{M}_{Q,J}(d^v)$. First of all, we observe that these objects form an open substack of both $\mathfrak{M}_X(v)$ and $\mathfrak{M}_{Q,J}(d^v)$, because of the following Lemma:

\begin{lemma}
Let the heart $\mathcal{A}\subset D^b(X)$ be either $\mathcal{C}$ or $\mathcal{K}$. Then for all $\mathcal{F}\in D^b(X\times S)$ the set $\{s\in S\ |\ \mathcal{F}_s\in\mathcal{A}\}$ is open.
\end{lemma}

\begin{proof}
This follows immediately from \cite[Rmk 3.11]{Toda08Modul}, after checking that both the hearts $\mathcal{C}$ and $\mathcal{K}$ have the so-called \emph{generic flatness} property: for $\mathcal{C}$ this reduces to the well-known fact that for any $\mathcal{F}\in\Coh_{\mathcal{O}_{X\times S}}$ there is an open nonempty set $U\subset S$ such that $\mathcal{F}{\restriction_U}$ is $U$-flat. For $\mathcal{K}$, the explanation is completely analogous: the equivalence \eqref{Eq: der eq families} induces a heart $\mathcal{K}_S\subset D^b(X\times S)$, and an object $\mathcal{F}\in\mathcal{K}_S$ maps to a sheaf in $\Coh_{\mathcal{O}_S}^{(Q,J)}$, which is locally free when restricted to an open nonempty subset $U\subset S$; by what we said after Eq.\ \eqref{Eq: der eq families}, this means that $\mathcal{F}{\restriction_U}$ is a flat family of objects of $\mathcal{F}$.
\end{proof}

Summing up, the moduli stacks $\mathfrak{M}_X(v)$ and $\mathfrak{M}_{Q,J}(v)$, parameterizing flat families of objects in $\mathcal{C}$ and $\mathcal{K}$ respectively, have a common open substack consisting of families of objects lying in the intersection $\mathcal{C}\cap\mathcal{K}$. In particular, we have thus an open substack $\mathfrak{M}'\subset\mathfrak{M}_{X,C_0}^\textup{t}(v)$ parameterizing framable sheaves which also belong to $\mathcal{K}$, i.e.\ those to which the description by linear data applies.

Our last problem will be to show that, under certain conditions, \emph{every} framable sheaf belongs to $\mathcal{K}$, so that $\mathfrak{M}'=\mathfrak{M}_{X,C_0}^\textup{t}(v)$ and the description by linear data is complete.

\subsection{Cohomology vanishings for framed sheaves on a movable line}\label{Cohom van fr sh}
To achieve the goal stated at the end of the previous subsection we need a vanishing result for the cohomologies of torsion-free sheaves which are trivial on the curve $C_0$:

\begin{proposition}\label{Prop: van fr sh mov lin}
Let $C_0\subset X$ be isomorphic to $\mathbb{P}^1$, and assume that the linear system $|C_0|$ has positive dimension, and that there exists have a line bundle $\mathcal{L}\in\Pic X$ such that $\cc_1(\mathcal{L})\cdot C_0=1$. Let $\mathcal{E}$ be a torsion-free sheaf whose restriction $\mathcal{E}{\restriction_{C_0}}$ is trivial. Then, for any divisor $D\in\Div(X)$ we have:
\begin{equation}
\begin{array}{ccc}
H^0(X;\mathcal{E}(D))=0 &\textup{if} &D\cdot C_0<0\,,\\
H^2(X;\mathcal{E}(D))=0 &\textup{if} &(D-K_X)\cdot C_0>0\,,
\end{array}
\end{equation}
where $K_X$ denotes the canonical divisor of $X$.
\end{proposition}

\begin{proof}
First of all, we have an exact sequence $0\to\mathcal{E}\to\mathcal{E}^{\vee\vee}\to\mathcal{Q}\to 0$ with $\mathcal{E}^{\vee\vee}$ locally free and $\mathcal{Q}$ supported on points. The long exact sequence in cohomology gives thus injections $H^\ell(X;\mathcal{E})\hookrightarrow H^\ell(X;\mathcal{E}^{\vee\vee})$ for $\ell=0,2$. Hence we may assume now on that $\mathcal{E}$ is locally free. 

Denote by $\mathfrak{d}:=|C_0|$ the linear system of $C_0$ and by $\mathfrak{d}_\textup{s}\subset\mathfrak{d}$ the smooth locus. By the genus formula, all the divisors in $\mathfrak{d}_\textup{s}$ are isomorphic to $\mathbb{P}^1$.

Now let $\mathfrak{d}_\textup{t}\subset\mathfrak{d}_\textup{s}$ consist of those curves $C$ such that $\mathcal{E}{\restriction_C}$ is trivial. We claim that this is an open subset: indeed, $\mathfrak{d}$ determines a family $Z\subset X\times\mathfrak{d}_\textup{s}$ of lines in $X$, flat over $\mathfrak{d}_\textup{s}$; we can thus apply the semicontinuity theorem to the bundle $\pr_X^*(\mathcal{E}\otimes\mathcal{L})$ to conclude that the set
\begin{equation*}
\mathfrak{d}_\textup{t}=\{C\in\mathfrak{d}_\textup{s}\ |\ \mathcal{E}{\restriction_C}\simeq\mathcal{O}_C^{\oplus r}\}=
\{C\in\mathfrak{d}_\textup{s}\ |\ h^0(C;\mathcal{E}{\restriction_C}(-1))=0\}=
\{C\in\mathfrak{d}_\textup{s}\ |\ h^0(Z_C;\pr_X^*(\mathcal{E}\otimes\mathcal{L}))=0\}
\end{equation*}
is open in $\mathfrak{d}$, having used the fact that a degree-zero bundle $\mathcal{V}$ on $\mathbb{P}^1$ is trivial if and only if $h^0(\mathbb{P}^1;\mathcal{V}(-1))=0$.

Then it is not hard to see that the union of the lines in $\mathfrak{d}_\textup{t}$ is a dense subset $U\subset X$ (because the set $\Omega\subset\mathfrak{d}^\vee$ of hyperlanes in $\mathfrak{d}$ intersecting $\mathfrak{d}_\textup{s}$ is open, and $U$ contains  its counterimage under the rational map $X\dashrightarrow\mathfrak{d}^\vee$ determined by $\mathfrak{d}$).

Now take $s\in H^0(X;\mathcal{E}(D))$ and consider, for all $C\in\mathfrak{d}_\textup{t}$ its restriction $s{\restriction_C}\in H^0(C;\mathcal{E}(D){\restriction_C})\simeq H^0(\mathbb{P}^1;\mathcal{O}_{\mathbb{P}^1}(C_0\cdot D))$. The last vector space vanishes if $C_0\cdot D<0$, so $s$ vanishes on $U$, and hence everywhere.

The second vanishing is obtained by Serre duality.
\end{proof}

\begin{remark}\label{Rmk: conseq vanish Prop}
Prop.\ \ref{Prop: van fr sh mov lin} has some useful immediate consequences:
	\begin{enumerate}
	\item If $C_0^2>0$, then $H^0(X;\mathcal{E}(-C_0))=0$; this fact can be used (see \cite[Lemma 5.2]{BaBrRa15Monad}) to conclude that for any $\mathcal{E},\mathcal{F}\in\mathfrak{M}^\textup{t}_{X,C_0}(v)$, the restriction map $\Hom(\mathcal{E},\mathcal{F})\to\Hom(\mathcal{E}{\restriction_{C_0}},\mathcal{F}{\restriction_{C_0}})$ is injective, and this implies in particular that any framed sheaf $(\mathcal{E},\phi)$ has trivial automorphism group.
	\item For any torsion-free sheaf $\mathcal{E}$ we have a surjection $H^2(X;\mathcal{E}^\vee\otimes\mathcal{E})\to\Ext^2(\mathcal{E},\mathcal{E})$, and the first vector space vanishes by the Proposition if $\mathcal{E}$ is trivial on $C_0$ and $C_0\cdot K_X<0$. Hence we deduce that the stack $\mathfrak{M}^\textup{t}_{X,C_0}(v)$ (and hence $\mathfrak{M}^\textup{fr}_{X,C_0}(v)$) is smooth if $C_0\cdot K_X<0$.
	\end{enumerate}
\end{remark}

Now let us apply the vanishing result to the setting of the previous subsection: we assume again that $X$ possesses a full strong exceptional sequence $\mathfrak{E}$, which now will consists of line bundles, say
\begin{equation*}
\mathfrak{E}=(\mathcal{O}_X(D_n),...,\mathcal{O}_X(D_0))
\end{equation*}
Again, we consider the equivalence $\Psi:D^b(X)\to D^b(Q;J)$ and the heart $\mathcal{K}\subset D^b(X)$ induced by $\mathfrak{E}$. By looking at explicit form \eqref{Eq: expl form der eq} of the equivalence $\Psi$, a sheaf $\mathcal{E}$ belongs to $\mathcal{K}$ if and only if it satisfies some cohomological vanishings, namely
\begin{equation}\label{Eq: cohom van bel to K}
H^\ell(X;\mathcal{E}(-D_i))=0\textup{ for }\ell=0,2\,,\ i=0,...,n\,.
\end{equation}
So these are the vanishings that we need to conclude, as discussed in the previous subsection, that the whole $\mathfrak{M}_{X,C_0}^\textup{t}(v)$ embeds in $\mathfrak{M}_{Q,J}(v)$ as an open substack. Hence Proposition \ref{Prop: van fr sh mov lin} immediately gives:

\begin{corollary}\label{Cor: Mt emb in MQ}
Let $C_0\subset X$ be isomorphic to $\mathbb{P}^1$, and assume that the linear system $|C_0|$ has positive dimension, and that there exists have a line bundle $\mathcal{L}\in\Pic X$ such that $\cc_1(\mathcal{L})\cdot C_0=1$. Suppose further that for any $i=0,...,n$ we have
\begin{equation*}
0<D_i\cdot C_0<-K_X\cdot C_0\,.
\end{equation*}
Then any torsion-free sheaf on $X$ which is trivial on $C_0$ belongs to the heart $\mathcal{K}$. In other words, the equivalence $\Psi$ embeds $\mathfrak{M}_{X,C_0}^\textup{t}(v)$ as an open substack $[U_v/G_{d^v}]\subset\mathfrak{M}_{Q,J}(d^v)$.
\end{corollary}

At this point we can easily end the proof of Theorem \ref{Thm: main thm}: by Prop.\ \ref{Prop: Mfr prin bun} $\mathfrak{M}_{X,C_0}^\textup{fr}(v)$ is a principal $\GL(r,\mathbb{C})$-bundle over $\mathfrak{M}_{X,C_0}^\textup{t}(v)\simeq[U_v/G_{d^v}]$, and thus, as explained at the end of Appendix \ref{App: prin bun on stack}, it is isomorphic to a quotient stack $[P_v/G_{d^v}]$, where $P_v\to U_v$ is a principal $\GL(r,\mathbb{C})$-bundle with a lift of the action of $G_{d^v}$. This stack is smooth because Remark \ref{Rmk: conseq vanish Prop} applies under the hypotheses of the Theorem, and the fact that it is represented by a scheme $\M_{X,C_0}^\textup{fr}(v)$ means that we have a principal $G_{d^v}$-bundle $P_v\to\M_{X,C_0}^\textup{fr}(v)$.

%
%
\section{Examples}\label{Examples}
First of all we recall from \cite[Thm 5.9]{HilPer11Exce} that if $X$ is $\mathbb{P}^2$, a Hirzebruch surface, or it is obtained from a Hirzebruch surface by blowing-up twice (possibly many points each time), then it has a full strong exceptional sequence made of line bundles.

In this section we study some examples of surfaces $X$ with a line $C_0$ and a collection $\mathfrak{E}$ satisfying all the hypotheses of Theorem \ref{Thm: main thm} (o slight modifications of it), and we compute the data $(Q,J)$ and $d^v$ which describe $\M^\textup{t}_{X,C_0}(v)$ according to the same Theorem. We consider the following examples: sheaves on $\mathbb{P}^2$ framed on a line (recovering the well-known description of \cite{Nakaj99Lect}) and sheaves on $\mathbb{P}^1\times\mathbb{P}^1$ framed on the diagonal divisor.

It should also be observed that some of the constructions available in the literature are obtained using generalized Beilinson spectral sequences coming from exceptional sequences which are not strong, so these cases cannot be included as examples of Theorem \ref{Thm: main thm}. A systematic treatment for them would require different techniques, which are beyond the scope of this paper.

\subsection{\texorpdfstring{$\mathbb{P}^2$}{P2}}
Let $X=\mathbb{P}^2$ with homogeneous coordinates $[x_0,x_1,x_2]$ and take $C_0=\ell_\infty$ to be the line $x_2=0$.

This was the first example in which $\M^\textup{fr}_{X,C_0}(v)$ was constructed via linear data in \cite{Nakaj99Lect}, building on the previous works \cite{Barth77Mod,Don84Ins}. Let us see how this is recovered in our framework: we have a full strong exceptional collection
\begin{equation*}
\mathfrak{E}=(E_2,E_1,E_0)=(\mathcal{O}(1),\uptau_{\mathbb{P}^2},\mathcal{O}(2))
\end{equation*}
where $\uptau_{\mathbb{P}^2}\simeq\Omhol_{\mathbb{P}^2}^1(3)$ is the tangent sheaf. The associated bound quiver $(Q,J)$ is
\begin{center}$Q$: \begin{tikzcd}
0\arrow[bend left=50]{r}{a_1}\arrow{r}{a_2}\arrow[bend right=50]{r}{a_3} &1\arrow[bend left=50]{r}{b_1}\arrow{r}{b_2}\arrow[bend right=50]{r}{b_3} &2
\end{tikzcd}\,,\quad $J=(b_ia_j+b_ja_i,\ i,j=1,2,3)$,\end{center}
and under the isomorphism $K_0(X)\simeq K_0(Q)$ induced by $\mathfrak{E}$, the class $v$ is sent to the dimension vector $d^v=(d^v_0,d^v_1,d^v_2)$ given by
\begin{equation*}
\begin{pmatrix}
d^v_0\\
d^v_1\\
d^v_2
\end{pmatrix}=
\begin{pmatrix}
1 &2 &-1\\
3 &3 &-2\\
1& 1& -1
\end{pmatrix}
\begin{pmatrix}
\rk v\\
\deg v\\
\chi(v)
\end{pmatrix}
\end{equation*}
(see \cite[\S5.1]{Maior17Mod}). In this case the collection $\mathcal{E}$ does not consist of line bundles only, so we cannot apply directly Theorem\ \ref{Thm: main thm}. However, the same conclusions hold, since the needed cohomological vanishings apply:

\begin{lemma}\cite[Lemma 2.3]{Nakaj99Lect}
If $\mathcal{E}\in\Coh_{\mathcal{O}_{\mathbb{P}^2}}$ is a torsion-free sheaf trivial on $C_0$, then $H^\ell(\mathbb{P}^2,E\otimes E_i^\vee)=0$ for $i=0,1,2$ and $\ell=0,2$, so $\mathcal{E}\in\mathcal{K}$.
\end{lemma}

\begin{proof}
For $i=0,2$ the vanishings follow directly from Prop.\ \ref{Prop: van fr sh mov lin}, since $0<D\cdot C_0<-K_{\mathbb{P}^2}\cdot C_0=3$ for $D=H,2H$, and we can take $\mathcal{L}=\mathcal{O}(H)$. For $i=1$, we observe that the restriction of $\uptau_{\mathbb{P}^2}$ to any line in $\mathbb{P}^2$ is isomorphic to $\mathcal{O}_{\mathbb{P}^1}(2)\oplus\mathcal{O}_{\mathbb{P}^1}(1)$ (see \cite[ch.\ 1, \S2.2]{OSS80Ve}), so the proof of Prop.\ \ref{Prop: van fr sh mov lin} continues to work if at the end we replace $\mathcal{O}(-D)$ with $\uptau_{\mathbb{P}^2}^\vee$, giving the desired vanishings.
\end{proof}

So the same conclusions of Theorem\ \ref{Thm: main thm} apply, and we get a principal $\GL(r,\mathbb{C})$-bundle $\M^\textup{fr}_{X,\ell_\infty}(v)\simeq [P_v/G_{d^v}]\to[U_v/G_{d^v}]\subset\mathfrak{M}_{Q,J}(d^v)$.

%
%

\begin{remark}\label{Rmk: mod fr sh P2 quiv var}
In fact, in this case one can also make $U_v$ explicit and simplify this description to prove that $\M_{\mathbb{P}^2,\ell_\infty}^\textup{fr}(v)\simeq[P_v/G_{d^v}]$ is isomorphic to the Nakajima quiver variety $\M_{0,-1}(k,r)$ associated to the Jordan quiver
\begin{center}\begin{tikzcd}[column sep=small]
\circ\ar[loop,swap,looseness=4]
\end{tikzcd}\,,\end{center}
where $k:=\cc_2(v)$. See \cite[\S2.1]{Nakaj99Lect}, \cite[\S5.6]{Ginz09Lec}.
\end{remark}



%
\subsection{\texorpdfstring{$\mathbb{P}^1\times\mathbb{P}^1$}{P1xP1}}
Take $X:=\mathbb{P}^1\times\mathbb{P}^1$ and let $C_0:=\Delta$ be the diagonal divisor (which is big and nef). Recall that $\Pic X=\mathbb{Z}H\oplus\mathbb{Z}F$, where $H=\{[1:0]\}\times\mathbb{P}^1$ and $F=\mathbb{P}^1\times \{[1:0]\}$, the intersection form is given by
\begin{equation*}
H^2=F^2=0\,,\quad H\cdot F=1
\end{equation*}
and the canonical divisor is $K_X=-2H-2F$. We also denote
\begin{equation*}
\mathcal{O}_X(a,b):=\mathcal{O}_X(aH+bF)=\mathcal{O}_{\mathbb{P}^1}(a)\boxtimes\mathcal{O}_{\mathbb{P}^1}(b)\,.
\end{equation*}
$X$ has a full strong exceptional collection
\begin{equation*}
\mathfrak{E}=(\mathcal{O}_X(1,0),\mathcal{O}_X(1,1),\mathcal{O}_X(2,0),\mathcal{O}_X(2,1))
\end{equation*}
which determines an equivalence $D^b(X)\simeq D^b(Q;J)$, where (see \cite[\S6]{Maior17Mod})
\begin{center}$Q$: \begin{tikzcd}
&1\arrow[shift left]{dr}{b^1_1}\arrow[shift right,swap]{dr}{b^1_2} &\\
0\arrow[shift left]{ur}{a^1_1}\arrow[shift right,swap]{ur}{a^1_2}\arrow[shift left]{dr}{a^2_1}\arrow[shift right,swap]{dr}{a^2_2} &&2\\
&3\arrow[shift left]{ur}{b^2_1}\arrow[shift right,swap]{ur}{b^2_2}&
\end{tikzcd}\,,\quad  $J=(b^1_ia^1_j+b^2_ja^2_i,\ i=1,2)$\end{center}
and, for $v\in K_0(X)$, the corresponding dimension vector $d^v=(d^v_0,d^v_1,d^v_2,d^v_3)$ is
\begin{equation*}
\begin{pmatrix}
d^v_0\\
d^v_1\\
d^v_2\\
d^v_3
\end{pmatrix}=
\begin{pmatrix}
1 &1 &2 &-1\\
2 &0 &2 &-1\\
1 &1 &1 &-1\\
1 &0 &1 &-1
\end{pmatrix}
\begin{pmatrix}
\rk v\\
\cc_1(v)_H\\
\cc_1(v)_F\\
\chi(v)
\end{pmatrix}\,,
\end{equation*}
where we wrote $\cc_1(v)=\cc_1(v)_H[H]+\cc_1(v)_F[F]$. $C_0$ and $\mathfrak{E}$ satisfy all the hypotheses of Theorem \ref{Thm: main thm}, as $\mathcal{L}$ can be chosen to be $\mathcal{O}_X(1,0)$ and for all $(a,b)\in\{(1,0),(1,1),(2,0),(2,1)\}$ we have
\begin{equation*}
0<(aH+bF)\cdot\Delta=a+b<-K_X\cdot\Delta=4\,.
\end{equation*}

\begin{remark}
Another possibility would be to consider sheaves which are framed on the cross $C_0:=H\cup F$: in this case we get all the required cohomological vanishings by applying Prop.\ \ref{Prop: van fr sh mov lin} to $H$ and $F$ separately. Again, the conclusions of Theorem \ref{Thm: main thm} hold, so we get as usual an open embedding $\mathfrak{M}^\textup{t}_{X,C_0}(v)\subset\mathfrak{M}_{Q,J}(d^v)$.

However, this case reduces to that of the previous subsection: the rational map $\mathbb{P}^2\dashrightarrow\mathbb{P}^1\times\mathbb{P}^1$ sending $[x_0:x_1:x_2]$ to $([x_0:x_2],[x_1:x_2])$ becomes an isomorphism after blowing up $p_1=[1:0:0]$ and $p_2=[0:1:0]$ in $\mathbb{P}^2$, and $q=([1:0],[1:0])$ in $\mathbb{P}^1\times\mathbb{P}^1$, giving a diagram
\begin{center}\begin{tikzcd}
&\Bl_{p_1,p_2}\mathbb{P}^2\simeq\Bl_q\mathbb{P}^1\times\mathbb{P}^1\arrow{dr}\arrow{dl}\\
\mathbb{P}^2 &&\mathbb{P}^1\times\mathbb{P}^1
\end{tikzcd}\end{center}
under which the divisors $\ell_\infty\subset\mathbb{P}^2$ and $C_0\subset\mathbb{P}^1\times\mathbb{P}^1$ are correspondent, so that $\M^\textup{fr}_{\mathbb{P}^1\times\mathbb{P}^1,C_0}(v)\simeq\M^\textup{fr}_{\mathbb{P}^2,\ell_\infty}(r,0,\chi)$ if $v=(r,0,0,\chi)$. In fact, it is not hard to describe the embedding $\mathfrak{M}^\textup{t}_{X,C_0}(v)\subset\mathfrak{M}_{Q,J}(d^v)$ explicitly and recover from that the realization of the moduli space as a quiver variety mentioned in Remark \ref{Rmk: mod fr sh P2 quiv var}.
\end{remark}
\appendix
\section{Principal bundles over stacks}\label{App: prin bun on stack}
In this Appendix we review the formalism on principal bundles on algebraic stacks as used in the paper, essentially following \cite[\S1]{BiSoWo12Root}.

Let $\mathfrak{X}$ be an algebraic $\mathbb{C}$-stack. To define a principal $G$-bundle over $\mathfrak{X}$ we could mimic the definition for schemes, i.e.\ by taking a representable morphism $\mathfrak{P}\to\mathfrak{X}$ with a $G$-action on $\mathfrak{P}$ which after base-change to an atlas $U\to\mathfrak{X}$ becomes equivariantly isomorphic to the trivial $G$-space $U\times G$. However, we prefer to avoid dealing with $G$-actions on stacks and recover instead the morphism $\mathfrak{P}\to\mathfrak{X}$ from the following definition:

\begin{definition}\label{Defn: prin G-bun stk}
A \emph{principal $G$-bundle} over $\mathfrak{X}$ is a morphism $\Phi:\mathfrak{X}\to BG$.
\end{definition}

This means that $\Phi$ sends a family $\mathcal{F}\in\mathfrak{X}(S)$ to a principal $G$-bundle $P_\mathcal{F}\to S$, and a Cartesian morphism $F:\mathcal{F}'\to\mathcal{F}$ over $f:S'\to S$ to a morphism
\begin{center}\begin{tikzcd}
P_{\mathcal{F}'}\arrow{r}{\Phi(F)}\arrow{d}\arrow[dr, phantom, "\ulcorner", very near start]
 &P_\mathcal{F}\arrow{d}\\
S'\arrow{r}{f} &S
\end{tikzcd}\end{center}
of principal $G$-bundles.

\begin{remarks}\ 
\begin{enumerate}
\item The principal $G$-bundles $P_\mathcal{F}\to S$, together with the isomorphisms $k^*P_\mathcal{F}\simeq P_{k^*\mathcal{F}}$ for any morphism $k:S'\to S$, completely characterize $\Phi$. In fact, it is enough to assign of such bundles only for smooth atlases $U\to\mathfrak{X}$ to determine $\Phi$ (see \cite[\S1.2]{BiSoWo12Root}).
\item If $X$ is a scheme, Def.\ \ref{Defn: prin G-bun stk} reduces to the usual one, as $\MOR(X,BG)\simeq BG(X)$ is the category of principal $G$-bundles on $X$.
\end{enumerate}\end{remarks}

From a principal bundle $\Phi:\mathfrak{X}\to BG$, we can canonically construct a new stack $\mathfrak{P}$:
\begin{equation*}
\begin{array}{c}
\Ob(\mathfrak{P})=\{(\mathcal{F},\sigma)\ |\ \mathcal{F}\in\mathfrak{X}(S)\,,\ \sigma:S\to P_\mathcal{F}\textup{ section of }\Phi(\mathcal{F})=(P_\mathcal{F}\to S)\}\\
\Mor_\mathfrak{P}((\mathcal{F}',\sigma'),(\mathcal{F},\sigma))=\{F\in\Mor_\mathfrak{P}(\mathcal{F}',\mathcal{F})\ | \Phi(F)\circ\sigma'=\sigma\circ f\}
\end{array}
\end{equation*}
(where $F$ is again a morphism over $f:S'\to S$), and we have the obvious forgetful functor $\mathfrak{P}\to\mathfrak{X}$.

\begin{definition}\label{Defn: tot sp pr bun stk}
$\mathfrak{P}$, together with the morphism $\mathfrak{P}\to\mathfrak{X}$, is called the \emph{total space} of the principal $G$-bundle $\Phi:\mathfrak{X}\to BG$.
\end{definition}
Notice that in this paper we often call the morphism $\mathfrak{P}\to\mathfrak{X}$ itself the ``principal $G$-bundle''.

\begin{example}
On a quotient stack $[X/G]$ we have a canonical principal $G$-bundle, given by the forgetful morphism $[X/G]\to BG$. Its total space can be identified with the canonical morphism $X\to[X/G]$.
\end{example}

\begin{lemma}\label{Lem: pullb fam prin bun stk}
Given a family $\mathcal{F}\in\mathfrak{X}(S)$ and its associated morphism $\alpha_\mathcal{F}:S\to\mathfrak{X}$, we have a 2-Cartesian diagram of stacks:
	\begin{center}\begin{tikzcd}
	P_\mathcal{F}\arrow{r}\arrow{d}\arrow[dr, phantom, "\ulcorner", very near start] &\mathfrak{P}\arrow{d}\\
	S\arrow{r}{\alpha_\mathcal{F}} &\mathfrak{X}
	\end{tikzcd}\end{center}
\end{lemma}

In particular, this shows that the morphism $\mathfrak{P}\to\mathfrak{X}$ is representable by schemes, it is smooth of relative dimension $\dim G$, and $\mathfrak{P}$ is an algebraic $\mathbb{C}$-stack.

\begin{proof}
Denote by $p:P_\mathcal{F}\to S$ the principal bundle $\Phi(\mathcal{F})$. First, recall (see e.g.\ \cite[2.2.2]{LauMor00Champs}) that objects of $S\times_\mathfrak{X}\mathfrak{P}$ over a scheme $S'$ are triples $(k,(\mathcal{F}',\rho),G)$, where $k:S'\to S$ is a morphism of schemes, $\mathcal{F}'\in\mathfrak{X}(S')$, $\rho:S'\to P_{\mathcal{F}'}$ a section of the principal bundle $\Phi(\mathcal{F}')$, and $G:k^*\mathcal{F}\overset{\simeq}{\to}\mathcal{F}'$ an isomorphism in $\mathfrak{X}(S')$. Second, note that a morphism $f:S'\to P_\mathcal{F}$ determines a unique section $\sigma':S'\to P'$ of the pull-back bundle $P'\to S'$ of $P_\mathcal{F}\to S$ under $p\circ f:S'\to S$, compatible with the pull-back diagram. Using this notation, the equivalence $P_\mathcal{F}\simeq S\times_\mathfrak{X}\mathfrak{P}$ is given by the functor
\begin{equation*}
\begin{array}{ccc}
P_\mathcal{F} &\longrightarrow &S\times_\mathfrak{X}\mathfrak{P}\\
S'\overset{f}{\to}P_\mathcal{F} &\longmapsto &(p\circ f,((p\circ f)^*\mathcal{F},\sigma'),\Id_{(p\circ f)^*\mathcal{F}})\\
S'_1\overset{h}{\to}S'_2 &\longmapsto &(h,\alpha_\mathcal{F}(h))
\end{array}\,.
\end{equation*}
\end{proof}

Finally, we focus on principal $G$-bundles on a quotient stack (we refer to \cite[\S1.3]{BiSoWo12Root} for details): let $X$ be a scheme with the action $\Gamma\curvearrowright X$ of an algebraic group. Consider the quotient $[X/\Gamma]$ and take a principal $G$-bundle $\Phi:[X/\Gamma]\to BG$ over it.

First, the canonical atlas $\nu:X\to[X/\Gamma]$ induces via $\Phi$ a principal $G$-bundle $p:P_\nu\to X$. Moreover, this carries an action $\Gamma\circlearrowright P_\nu$ making $p$ $\Gamma$-equivariant, and this action completely determines $\Phi$ . For any family in $[X/\Gamma](S)$, which is a principal $\Gamma$-bundle $Q\to S$ with a $\Gamma$-equivariant map $f:Q\to X$, we can take the pullback $G$-bundle $P_{\nu\circ f}\to Q$ of $p$ along $f$: this is also a $\Gamma$-equivariant $G$-bundle, and thus it descends via the $\Gamma$-bundle $Q\to S$ to a principal $G$-bundle $P_\phi\to S$. Putting all together, we have a 2-commutative diagram
\begin{center}\begin{tikzcd}[row sep=tiny, column sep=tiny]
P_{\nu\circ f}\arrow{rr}\arrow{dr}\arrow{dd}\arrow[drrr, phantom, "\roangle", very near start]\arrow[dddr, phantom, "\rroangle", very near start] &&P_\nu\arrow{dr}{p}\\
&Q\arrow{rr}{f}\arrow{dd}\arrow[ddrr, phantom, "\ulcorner", very near start] &&X\arrow{dd}{\nu}\\
P_\phi\arrow{dr}\\
&S\arrow{rr}{\phi} && {[X/\Gamma]}
\end{tikzcd}\end{center}
where the squares are Cartesian, the vertical lines are principal $\Gamma$-bundles and the diagonal maps are principal $G$-bundles. Moreover, in the above diagram we have 1-1 correspondences between sections of $P_\phi\to S$, $\Gamma$-equivariant sections of $P_{\nu\circ f}\to Q$, and $\Gamma$-equivariant maps $Q\to P_\nu$. Using this fact it is easy to see that the total space of $\Phi$ can be identified with the morphism
\begin{equation*}
[P_\nu/\Gamma]\longrightarrow[X/\Gamma]
\end{equation*}
induced by $p$.


\bibliographystyle{alpha}

\bibliography{Bib-Moduli,Bib-CatHomAlg,Bib-VarPap,Bib-AlgGeo}

\begin{thebibliography}{BMW12}

\bibitem[Bar77]{Barth77Mod}
Wolf Barth.
\newblock Moduli of vector bundles on the projective plane.
\newblock {\em Invent. Math.}, 42(1):63--91, 1977.

\bibitem[BBR15]{BaBrRa15Monad}
Claudio Bartocci, Ugo Bruzzo, and Claudio~L.S. Rava.
\newblock Monads for framed sheaves on {Hirzebruch} surfaces.
\newblock {\em Adv. Geom.}, 15(1):55--76, 2015.

\bibitem[BCZ17]{BaCrZh17Nef}
Arend Bayer, Alastair Craw, and Ziyu Zhang.
\newblock Nef divisors for moduli spaces of complexes with compact support.
\newblock {\em Selecta Mathematica}, 23(2):1507--1561, 2017.

\bibitem[BM11]{BruMar11Mod}
Ugo Bruzzo and Dimitri Markushevich.
\newblock Moduli of framed sheaves on projective surfaces.
\newblock {\em Doc. Math.}, 16:399--410, 2011.

\bibitem[BMW12]{BiSoWo12Root}
Indranil Biswas, Souradeep Majumder, and Michael~Lennox Wong.
\newblock Root stacks, principal bundles and connections.
\newblock {\em Bull. Sci. math}, 136:369--398, 2012.

\bibitem[Don84]{Don84Ins}
Simon~K. Donaldson.
\newblock Instantons and geometric invariant theory.
\newblock {\em Comm. Math. Phys.}, 93(4):453--460, 1984.

\bibitem[Gin09]{Ginz09Lec}
Victor Ginzburg.
\newblock Lectures on {Nakajima}’s quiver varieties.
\newblock {\em arXiv preprint arXiv:0905.0686}, 2009.

\bibitem[Gro63]{Groth63EGAIII2}
Alexandre Grothendieck.
\newblock {\em {\'E}l{\'e}ments de g{\'e}om{\'e}trie alg{\'e}brique II. Etudes
  cohomologiques et faisceaux coherents (Secondie Partie)}.
\newblock Publ. Math. IHES, 1963.

\bibitem[Hen14]{Henni14Monads}
Amar~Abdelmoubine Henni.
\newblock Monads for framed torsion-free sheaves on multi-blow-ups of the
  projective plane.
\newblock {\em Internat. J. Math.}, 25(01):1450008, 2014.

\bibitem[HL95]{HuyLeh95Framed}
Daniel Huybrechts and Manfred Lehn.
\newblock Framed modules and their moduli.
\newblock {\em Internat. J. Math.}, 6(02):297--324, 1995.

\bibitem[HP11]{HilPer11Exce}
Lutz Hille and Markus Perling.
\newblock Exceptional sequences of invertible sheaves on rational surfaces.
\newblock {\em Compositio Mathematica}, 147(4):1230--1280, 2011.

\bibitem[Kin89]{King89Inst}
Alastair King.
\newblock {\em Instantons and holomorphic bundles on the blown-up plane}.
\newblock PhD thesis, Thesis, Oxford University, 1989.

\bibitem[LMB00]{LauMor00Champs}
G{\'e}rard Laumon and Laurent Moret-Bailly.
\newblock {\em Champs alg{\'e}briques}, volume~39.
\newblock Springer, 2000.

\bibitem[Mai17]{Maior17Mod}
Andrea Maiorana.
\newblock Moduli of semistable sheaves as quiver moduli.
\newblock {\em arXiv preprint arXiv:1709.05555}, 2017.

\bibitem[Mai18]{Maior18Mod}
Andrea Maiorana.
\newblock {\em Moduli of sheaves, quiver moduli, and stability}.
\newblock PhD thesis, SISSA, \textit{available at
  \href{https://iris.sissa.it/bitstream/20.500.11767/82314/1/Andrea\%20Maiorana\%20-\%20Moduli\%20of\%20sheaves\%2C\%20quiver\%20moduli\%2C\%20and\%20stability\%20\%28PhD\%20thesis\%29.pdf}{https://iris.sissa.it}},
  2018.

\bibitem[Mar76]{Maruy76Openn}
Masaki Maruyama.
\newblock Openness of a family of torsion free sheaves.
\newblock {\em Journal of Mathematics of Kyoto University}, 16(3):627--637,
  1976.

\bibitem[Nak99]{Nakaj99Lect}
Hiraku Nakajima.
\newblock {\em Lectures on Hilbert schemes of points on surfaces}, volume~18.
\newblock American Mathematical Society Providence, RI, 1999.

\bibitem[OSS80]{OSS80Ve}
Christian Okonek, Michael Schneider, and Heinz Spindler.
\newblock {\em Vector bundles on complex projective spaces}, volume~3.
\newblock Springer, 1980.

\bibitem[Per09]{Perl09Examp}
Markus Perling.
\newblock Examples for exceptional sequences of invertible sheaves on rational
  surfaces.
\newblock {\em arXiv preprint arXiv:0904.0529}, 2009.

\bibitem[Tod08]{Toda08Modul}
Yukinobu Toda.
\newblock Moduli stacks and invariants of semistable objects on k3 surfaces.
\newblock {\em Advances in Mathematics}, 217(6):2736--2781, 2008.

\bibitem[TU10]{TodUeh10Tilt}
Yukinobu Toda and Hokuto Uehara.
\newblock Tilting generators via ample line bundles.
\newblock {\em Adv. Math.}, 223(1):1--29, 2010.

\end{thebibliography}


\end{document}